\documentclass[11pt,reqno]{amsart}

\usepackage{amsmath,amssymb,amsthm,amsfonts,mathrsfs,eucal,stmaryrd,tikz-cd,yhmath,datetime}

\usepackage[utf8]{inputenc}

\numberwithin{equation}{section}

\usepackage[a4paper,margin=1.2in]{geometry}
                		
\usepackage{graphicx}				
										
\usepackage{amssymb}

\usepackage[utf8]{inputenc}

\usepackage[polish, english]{babel}

\usepackage[T1]{fontenc}

\usepackage{mathtools}

\usepackage{tikz-cd}

\usepackage{amsthm}

\usepackage[toc]{appendix}

\usepackage{calligra}

\usepackage{hyperref}

\usepackage{array}

\usepackage{makecell}

\newtheorem{thm}{Theorem}[section]
\newtheorem{lem}[thm]{Lemma}

\theoremstyle{remark}
\newtheorem{rmk}[thm]{Remark}
\newtheorem{ex}[thm]{Example}

\newcommand{\cO}{{\mathcal O}}

\newcommand{\PP}{{\mathbb P}}

\newcommand{\QQ}{{\mathbb Q}}

\newcommand{\ZZ}{{\mathbb Z}}

\newcommand{{\D}}{{\mathscr{D}_{X}}} 

\newcommand{\eE}{{\mathscr{E}}}
\newcommand{\Ox}{{\mathscr{O}_{X}}}

\title{Rank one summands of Frobenius pushforwards of line bundles on G/P} 

\author{Feliks Rączka}

\address{Institute of Mathematics, Polish Academy of Sciences, ul.\ Śniadeckich 8,
    \newline\indent 00-656 Warsaw, Poland
  }

\email{fraczka@impan.pl}

\begin{document}

\begin{abstract}
Let $X=G/P$ be a partial flag variety, where $G$ is a semi-simple, simply connected algebraic group defined over an algebraically closed field $K$ of positive characteristic. Let $\mathsf{F}\colon X\to X$ be the absolute Frobenius morphism. Given a line bundle $\mathscr{L}$ on $X$ and an integer $r\geq1$, we describe all line bundles that are direct summands of the pushforward $\mathsf{F}_{*}^{r}\mathscr{L}$. For $\mathscr{L}$ corresponding to a dominant weight, we also compute, for $r$ sufficiently large, the multiplicity of $\Ox$ as a summand of $\mathsf{F}_{*}^{r}\mathscr{L}$. As an application we answer a question of Gros--Kaneda.
\end{abstract}

\maketitle

\section{Introduction}\label{Introduction}

Let $K$ be an algebraically closed field of characteristic $p>0$, let $G$ be a semi-simple, simply connected algebraic group over $K$, and let $P\subset G$ be a parabolic subgroup. Let $\mathsf{F}\colon G/P\to G/P$ be the absolute Frobenius morphism and denote by $\mathsf{F}^{r}$ the composition of $r$ absolute Frobenii. As we explain in greater detail in Section \ref{Section: Motivation}, it is an important and difficult problem to determine, for a given $\mathscr{L}\in\textnormal{Pic}(G/P)$, the decomposition of $\mathsf{F}^{r}_{*}\mathscr{L}$ into a direct sum of indecomposable vector bundles (following \cite[Section 4, Definition]{Atiyah} we call such a decomposition the \textit{Remak decomposition}). The goal of this paper is to provide two results about the line bundles that appear in this decomposition. First, the line bundles on $G/P$ are parameterized by the characters of $P$, and in Theorem \ref{SummandOfLThm} we classify the line bundles that are direct summands of $\mathsf{F}^{r}_{*}\mathscr{L}$ in terms of this identification (for every $\mathscr{L}\in\textnormal{Pic}(G/P)$). Second, we try to address the problem of computing the multiplicities of these line bundles as summands of $\mathsf{F}^{r}_{*}\mathscr{L}$. This task seems to be quite difficult and we do not have a complete solution, but in Theorem \ref{MultiplicityofOxThm} we show how to compute the multiplicity of the structure sheaf if $\mathscr{L}$ corresponds to a dominant weight that (in a suitable sense) is smaller than $p^{r}\rho$, where $\rho$ is half the sum of the positive roots. In particular, this result shows that if $\mathscr{L}$ corresponds to a dominant weight, then $H^{0}(G/P, \mathscr{L})\otimes\mathscr{O}_{G/P}$ is a direct summand of $\mathsf{F}^{r}_{*}\mathscr{L}$ for all $r\gg0$.

\medskip
We now formalize the discussion from the previous paragraph. We fix a maximal torus $T\subset G$ and a Borel subgroup $B\subset G$ such that $T\subset B\subset P$. We also fix the set $S$ of simple roots. We let $X(T)$ (resp. $X^{\vee}(T)$) be the group of characters (resp. co-characters) of $T$ and we write $\langle-,-\rangle:X(T)\times X^{\vee}(T)\to\ZZ$ for the canonical perfect pairing. We write $X(P)$ for the character group of $P$, which parametrizes $G$-equivariant line bundles on $G/P$. For $\mu\in X(P)$, we denote by $\mathscr{L}^{P}(\mu)$ the corresponding line bundle. We write $\rho_{P}$ for the unique element of $\QQ\otimes_{\ZZ}X(P)$ such that $\omega_{G/P}=\mathscr{L}^{P}(-2\rho_{P})$. First, we prove the following theorem.

\begin{thm}\label{SummandOfLThm}
The following conditions are equivalent for $\mu,\lambda\in X(P)$.
\begin{enumerate}
    \item $\mathscr{L}^{P}(\lambda)$ is a direct summand of $\mathsf{F}^{r}_{*}\mathscr{L}^{P}(\mu)$.

    \item The inequality $0\leq \langle \mu-p^{r}\lambda,\alpha^{\vee}\rangle\leq (p^{r}-1)\langle 2\rho_{P},\alpha^{\vee
    }\rangle$ holds for all $\alpha\in S$.
\end{enumerate}
\end{thm}

\begin{rmk}
By \cite[Chapter 2.2, Exercises (3)-(4)]{Brion-Kumar}, the forgetful map $\textnormal{Pic}_{G}(G/P)\to\textnormal{Pic}(G/P)$ is an isomorphism, so Theorem \ref{SummandOfLThm} indeed describes all line bundles in the Remak decomposition of $\mathsf{F}^{r}_{*}\mathscr{L}^{P}(\mu)$.
On the other hand, in general, $\mathsf{F}^{r}_{*}\mathscr{L}^{P}(\mu)$ will also have indecomposable direct summands that are not line bundles. If $X$ is a smooth, projective $K$-variety, then by a result of P.\ Achinger \cite{Achinger1} $\mathsf{F}_{*}\mathscr{L}$ is a direct sum of line bundles for every $\mathscr{L}\in\textnormal{Pic}(X)$ if and only if $X$ is toric. For $X=G/P$ this is the case if and only if $X$ is a product of projective spaces.
\end{rmk}

\begin{rmk}
Recently, Cai--Krylov \cite{Cai-Krylov} studied Frobenius pushforwards of line bundles on wonderful compactification and (among other things) they obtained a description of rank one summands of $\mathsf{F}^{r}_{*}\mathscr{L}$. While their setting is different from ours, there are some similarities. For example, the proof of (1)$\implies$(2) in our Theorem \ref{SummandOfLThm} is essentially the same as the proof of Corollary 3.5.1 in \textit{loc.cit}.
\end{rmk}

Second, we study the multiplicity of $\mathscr{O}_{G/P}$ as a direct summand of $\mathsf{F}^{r}_{*}\mathscr{L}^{P}(\mu)$. We prove the following.
\begin{thm}\label{MultiplicityofOxThm}
Let $\mu\in X(P)$. If $0\leq \langle \mu,\alpha^{\vee}\rangle\leq p^{r}-1$ holds for all $\alpha\in S$, then $H^{0}(G/P,\mathscr{L}^{P}(\mu))\otimes\mathscr{O}_{G/P}$ is a direct summand of $\mathsf{F}^{r}_{*}\mathscr{L}^{P}(\mu)$. 
\end{thm}

\begin{rmk}
Since $\mathsf{F}$ is affine, it follows that $H^{0}(G/P,\mathsf{F}_{*}\mathscr{L})=H^{0}(G/P,\mathscr{L})$ (cf. Remark \ref{Frobenius_Affinity_Remark}). Therefore, in the situation of Theorem \ref{MultiplicityofOxThm}, $H^{0}(G/P,\mathscr{L}^{P}(\mu))\otimes\mathscr{O}_{G/P}$ is a maximal free summand of $\mathsf{F}^{r}_{*}\mathscr{L}^{P}(\mu)$. It follows that the multiplicity of $\Ox$ is given by $\dim_{K}H^{0}(G/P,\mathscr{L}^{P}(\mu))$, so it may be calculated by the well-known dimension formula of Weyl. 
\end{rmk}

\begin{rmk}
Let $P=B$ and $\mu=(p^{r}-1)\rho$. By a classical result of by H.\ H.\ Andersen \cite{Andersen_Frobenius} and W.\ J.\ Haboush \cite{Haboush_Kempf} we have $\mathsf{F}^{r}_{*}\mathscr{L}^{B}((p^{r}-1)\rho)\simeq H^{0}(G/B,\mathscr{L}^{B}((p^{r}-1)\rho))\otimes\mathscr{O}_{G/B}$. Theorem \ref{MultiplicityofOxThm} may be seen as a generalization of this result.
\end{rmk}

\medskip

\noindent
\textbf{Acknowledgments}
This work was supported by the project KAPIBARA funded by the European Research Council (ERC) under the European Union's Horizon 2020 research and innovation programme (grant agreement No 802787). I thank P.\ Achinger, N.\ Deshmukh, and A.\ Langer for stimulating discussions.

\section{Motivation}\label{Section: Motivation}

Before we proceed with the proof of Theorems \ref{SummandOfLThm} and \ref{MultiplicityofOxThm}, let us motivate our work by explaining how these results fit into a bigger picture. First, we mention that it is quite rare to know the Remak decomposition of $\mathsf{F}_{*}^{r}\mathscr{L}$ for a given smooth projective $X$ and for all $\mathscr{L}\in \textnormal{Pic}(X)$. 
Such decompositions are known if $X$ is either a toric variety or a quadric hypersurface in some projective space. In the toric case $\mathsf{F}_{*}^{r}\mathscr{L}$ is a direct sum of line bundles by the work of J.\ F.\ Thomsen \cite{Thomsen} (see also \cite{Bogvad} by R.\ B{\o}gvad). P.\ Achinger derived from this result a combinatorial description of the indecomposable summands of $\mathsf{F}_{*}^{r}\mathscr{L}$ (and their multiplicities) for all $\mathscr{L}$ in \cite{Achinger1}. In the case of quadrics, the description of $\mathsf{F}_{*}^{r}\mathscr{L}$ follows from the well-known classification of arithmetically Cohen--Macaulay vector bundles on these varieties. The decomposition of $\mathsf{F}_{*}^{r}\mathscr{L}$ was first described by A.\ Langer in \cite{Langer} and later refined by P.\ Achinger in \cite{Achinger2}. 
\medskip

Let us now focus on the case where $X=G/P$ is a partial flag variety. If $X$ is neither a product of projective spaces nor a quadric (hence it is not covered by the discussion in the previous paragraph), then not much is known about the direct summands of $\mathsf{F}^{r}_{*}\mathscr{L}$ for general $\mathscr{L}\in \textnormal{Pic}(X)$. On the other hand, for some special $\mathscr{L}$ the decomposition of $\mathsf{F}^{r}_{*}\mathscr{L}$ is known. One example is the aforementioned theorem of Andersen--Haboush which, among other things, provides a simple proof of Kempf's vanishing theorem for line bundles on $G/B$. It is also an interesting problem to study the Remak decomposition $\mathsf{F}_{*}^{r}\Ox$. On the one hand, such a decomposition allows to determine whenever $\mathsf{F}_{*}^{r}\Ox$ is a tilting generator of the derived category $D^{b}(X)$  (see, for example, the work of Hashimoto--Kaneda--Rumynin \cite{Hashimoto-Kaneda-Rumynin} for the case of $\textnormal{SL}_{3}/B$, Kandeda,\ M. \cite{Kaneda_1} for the case of type $G_{2}$, and Raedschelders--\v{S}penko--Van den Bergh \cite{Readschelders-Spenko-VdBergh_1}, \cite{Readschelders-Spenko-VdBergh_2} for the case of the grassmannian $\textnormal{Grass}(2,n)$). On the other hand, it is well known that on a flag variety the vanishing of $\textnormal{Ext}_{\Ox}^{i}(\mathsf{F}_{*}^{r}\Ox,\mathsf{F}_{*}^{r}\Ox)$ for all $r,i>0$ implies the $D$-affinity of $X$ (see, for example, the work of B.\ Haastert \cite{Haastert}, Kashiwara--Lauritzen \cite{Kashiwara-Lauritzen}, A.\ Langer \cite{Langer}, and A.\ Samokhin \cite{Samokhin-D-affinity_1} for the further discussion on $D$-affinity of flag varieties in positive characteristic). Therefore, the Remak decomposition of $\mathsf{F}^{r}_{*}\Ox$ carries a lot of information about the geometry of $X$.
\medskip

In connection with the above discussion, we include in the text two examples concerning $\mathsf{F}_{*}^{r}\Ox$, where $X=G/B$ is a full flag variety. In Example \ref{Ox-example_1}, we describe the line bundles that are direct summands of $\mathsf{F}_{*}^{r}\Ox$. In particular, $\mathscr{L}^{B}(-\rho)$ is such a summand, and in Example \ref{Ox-example_2} we compute its multiplicity. This answers the question posed by Gros--Kaneda at the end of \cite{Gros-Kaneda}.

\section{Preliminaries on flag varieties}

To prove Theorems \ref{SummandOfLThm} and \ref{MultiplicityofOxThm}, we need the following well-known facts from representation theory and algebraic geometry.
\medskip

We keep the notation and the assumptions from Section \ref{Introduction}. In what follows, we try to be consistent with the notation used in Jantzen's monograph \cite{Jantzen_book}. Recall that $X(T)$ is the group of characters of $T$. The set of dominant weights is
\[
X(T)_{+}\stackrel{\textnormal{def.}}{=}\left\{\mu\in X(T):\langle \mu,\alpha^{\vee}\rangle\geq 0\textnormal{ for all }\alpha\in S\right\}.
\]
We have $X(T)=X(B)$. If $B\subset P$ is a parabolic subgroup, then $P$ is determined by a subset $I\subset S$ of the set of simple roots. This allows to realize $X(P)$ as a subgroup of $X(T)$
\[
X(P)=
\left\{
\mu\in X(T):\langle \mu,\alpha^{\vee}\rangle=0\textnormal{ for all }\alpha\in I
\right\}
\]
(see \cite[II, Section 1.18, Fromula (4)]{Jantzen_book}). We set
\[
\rho\stackrel{\textnormal{def.}}{=}\frac{1}{2}\sum_{\alpha\in R_{+}}\alpha\in \frac{1}{2}X(T)
\]
($R_{+}$ is the set of positive roots), and more generally,
\[
\rho_{P}\stackrel{\textnormal{def.}}{=}\frac{1}{2}\sum_{\alpha\in R_{+}\setminus R_{I}}\alpha\in \frac{1}{2}X(T)
\]
($R_{I}=R_{+}\cap\ZZ I$). We denote
\[
H^{0}(\mu)\stackrel{\textnormal{def}.}{=}H^{0}(G/B,\mathscr{L}^{B}(\mu)).
\]
By \cite[II, Section 4.6, Proposition]{Jantzen_book}
\begin{equation}\label{G/BtoG/P_Setions}
H^{i}(G/B,\mathscr{L}^{}(\mu))=H^{i}(G/P,\mathscr{L}^{P}(\mu))\quad(i\geq0,\ \mu\in X(P)).
\end{equation}
It is well known \cite[II, Section 2.6]{Jantzen_book} that
\begin{equation}\label{Non-zeroH0}
H^{0}(\mu)\neq0\iff \mu\in X_{+}(T).
\end{equation}
We will also need the nontrivial but equally well known fact \cite[II, Section 14.20]{Jantzen_book} that for $\mu,\lambda\in X_{+}(T)$ the cup product
\begin{equation}\label{SurjectiveCupProduct}
H^{0}(\mu)\otimes H^{0}(\lambda)\to H^{0}(\mu+\lambda)
\end{equation}
is surjective.
\medskip

Let us now recall some basic facts about vector bundles over $X=G/P$.  We write $\omega_{X}$ for the canonical line bundle. We have \cite[II, Section 4.2, Formula (6)]{Jantzen_book}
\begin{equation}\label{LittleOmega}
\omega_{X}=\mathscr{L}^{P}(-2\rho_{P}).
\end{equation}
Since the Frobenius morphism is affine, the relative version of Serre's duality gives, for a vector bundle $\eE$ over $X$
\begin{equation}\label{Frobenius-Serre duality}
(\mathsf{F}^{r}_{*}\eE)^{\vee}=\left(\mathsf{F}_{*}^{r}(\eE^{\vee}\otimes\omega_{X})\right)\otimes\omega_{X}^{\vee}.
\end{equation}
For any line bundle $\mathscr{L}$ we have 
\begin{equation}\label{FrobeniusPullback}
(\mathsf{F}^{r})^{*}\mathscr{L}=\mathscr{L}^{\otimes p^{r}}.
\end{equation}
From (\ref{Frobenius-Serre duality}), (\ref{FrobeniusPullback}), and the projection formula, we obtain
\begin{equation}\label{Frobenius-Serre duality2}
(\mathsf{F}^{r}_{*}\eE)^{\vee}=\mathsf{F}^{r}_{*}(\eE^{\vee}\otimes\omega_{X}^{\otimes(1-p^{r})}).
\end{equation}
Since $\mathscr{L}^{P}(\mu)^{\vee}=\mathscr{L}^{P}(-\mu)$, we can combine (\ref{LittleOmega}) and (\ref{Frobenius-Serre duality2}) to obtain
\begin{equation}\label{LB-Frobenius duality}
(\mathsf{F}^{r}_{*}\mathscr{L}^{P}(\mu))^{\vee}=\mathsf{F}^{r}_{*}\mathscr{L}^{P}(2(p^{r}-1)\rho_{P}-\mu).
\end{equation}
Finally, we recall from \cite[Theorem 2.2.5]{Brion-Kumar} that $X$ is $F$-split, i.e., $\Ox$ is a direct summand of $\mathsf{F}_{*}\Ox$ (and therefore a direct summand of $\mathsf{F}^{r}_{*}\Ox$ for all $r\geq1$).

\section{Rank one summands of $\mathsf{F}_{*}\mathscr{L}$}

In this section, we prove Theorem \ref{SummandOfLThm}. We keep the notation and the assumptions from Section \ref{Introduction}. We let $X=G/P$ be a partial flag variety.

\begin{lem}\label{DominanceSplittingLemma}
Assume that $\mu_{1},\ \mu_{2},\ \mu_{1}-\mu_{2}\in X_{+}(T)\cap X(P)$. If $\Ox$ is a direct summand of $\mathsf{F}^{r}_{*}\mathscr{L}^{P}(\mu_{1})$ then it is also a direct summand of $\mathsf{F}^{r}_{*}\mathscr{L}^{P}(\mu_{2})$.
\end{lem}

\begin{proof}
Note that $\Ox$ is a direct summand of a vector bundle $\eE$ if and only if there exists a global section $s\in H^{0}(X,\eE)$ and an $\Ox$-linear map $\psi_{s}\colon \eE\to \Ox$ with $\psi_{s}(s)\neq0$. So, assume that we have such a section $s\in H^{0}(X,\mathsf{F}^{r}_{*}\mathscr{L}^{P}(\mu_{1}))=H^{0}(\mu_{1})$ and such a morphism $\psi_{s}
\colon\mathsf{F}^{r}_{*}\mathscr{L}^{P}(\mu_{1})\to\Ox$. 
The key observation is that because of (\ref{SurjectiveCupProduct}) we may write
\[
s=\sum s_{i}\otimes t_{i};\quad s_{i}\in H^{0}(\mu_{2}),\ t_{i}\in H^{0}(\mu_{1}-\mu_{2}).
\]
Since $\psi_{s}(s)\neq 0$, it follows that for some index $i$ we have $\psi_{s}(s_{i}\otimes t_{i})\neq 0$. Now, consider the $\Ox$-linear map
\begin{equation}\label{MultiplicationByaSection}
\varphi':\mathscr{L}^{P}(\mu_{2})\to\mathscr{L}^{P}(\mu_{1});\quad u\mapsto u\otimes t_{i}.
\end{equation}
This induces an $\Ox$-linear map
\[
\varphi=\mathsf{F}^{r}_{*}\varphi':\mathsf{F}^{r}_{*}\mathscr{L}^{P}(\mu_{2})\to\mathsf{F}^{r}_{*}\mathscr{L}^{P}(\mu_{1}),
\]
which at the level of global sections is still given by the formula (\ref{MultiplicationByaSection}). By construction
\[
(\psi_{s}\circ\varphi)(s_{i})=\psi_{s}(s_{i}\otimes t_{i})\neq 0,
\]
which shows that $\Ox$ is indeed a direct summand of $\mathsf{F}^{r}_{*}\mathscr{L}^{P}(\mu_{2})$.
\end{proof}

\noindent
We are ready to prove the Theorem \ref{SummandOfLThm}.

\begin{proof}[Proof of Theorem \ref{SummandOfLThm}]
It follows from the projection formula that $\mathscr{L}^{P}(\lambda)$ is a direct summand of $\mathsf{F}^{r}_{*}\mathscr{L}^{P}(\mu)$ if and only if $\Ox$ is a direct summand of $\mathsf{F}^{r}_{*}\mathscr{L}^{P}(\mu-p^{r}\lambda)$. After replacing $\mu-p^{r}\lambda$ with $\mu$, this observation reduces the proof of Theorem \ref{SummandOfLThm} to showing that the following statements are equivalent.
\begin{enumerate}
    \item[(1')] $\Ox$ is a direct summand of $\mathsf{F}^{r}_{*}\mathscr{L}^{P}(\mu)$.
    \item[(2')] The inequality $0\leq\langle\mu,\alpha^{\vee}\rangle\leq (p^{r}-1)\langle2 \rho_{P},\alpha^{\vee}\rangle$
    holds for any $\alpha\in S$.
\end{enumerate}

\noindent
(1')$\implies$(2'). If $\Ox$ is a direct summand of a vector bundle $\eE$ then it is also a direct summand of $\eE^{\vee}$, so both $\eE$ and $\eE^{\vee}$ have a non-zero global section. We have
\[
H^{0}(X,\mathsf{F}^{r}_{*}\mathscr{L}^{P}(\mu))=H^{0}(\mu),
\]
and
\[
H^{0}(X,\mathsf{F}^{r}_{*}\mathscr{L}^{P}(\mu)^{\vee})=H^{0}(2(p^{r}-1)\rho_{P}-\mu)
\]
by the affinity of Frobenius, (\ref{G/BtoG/P_Setions}), and (\ref{LB-Frobenius duality}). The result follows from (\ref{Non-zeroH0}).
\medskip

\noindent
(2')$\implies$(1'). Assume that $\mu$ satisfies (2'). Then 
\[
2(p^{r}-1)\rho_{P},\ \mu,\  2(p^{r}-1)\rho_{P}-\mu\in X_{+}(T)\cap X(P),
\]
so by Lemma \ref{DominanceSplittingLemma} we only have to show that $\mathsf{F}^{r}_{*}\mathscr{L}(2(p^{r}-1)\rho_{P})$ has $\Ox$ as a direct summand. By (\ref{LB-Frobenius duality}) we have
\begin{equation*}
\mathsf{F}^{r}_{*}\mathscr{L}(2(p^{r}-1)\rho_{P})=(\mathsf{F}^{r}_{*}\Ox)^{\vee},
\end{equation*}
so the claim follows from the fact that $X$ is $F$-split.
\end{proof}

\begin{ex}\label{Ox-example_1}
Let $X=G/B$ be the full flag variety. We will describe all line bundles that appear as direct summands of $\mathsf{F}^{r}_{*}\Ox$. Let $S=\{\alpha_{1},\dots,\alpha_{n}\}$, and let  $\omega_{1},\dots,\omega_{n}\in X(T)$ be the fundamental weights (that is, $\langle \omega_{i},\alpha_{j}^{\vee}\rangle=\delta_{ij}$). Then $\rho=\sum_{i=1}^{n}\omega_{i}$ and therefore 
\begin{equation}\label{RhoPairing}
\langle \rho,\alpha_{j}^{\vee}\rangle=1\qquad(1\leq j\leq n).
\end{equation}
It follows that $\mathscr{L}^{B}(\lambda)$ is a direct summand of $\mathsf{F}_{*}^{r}\Ox$ if and only if 
\begin{equation}\label{SummandsOfOX}
-\lambda=\omega_{i_{1}}+\omega_{i_{2}}+\dots+\omega_{i_{m}}\qquad(\ 1\leq i_{1}<i_{2}<\dots<i_{m}\leq n).
\end{equation}
Indeed, by Theorem \ref{SummandOfLThm} and (\ref{RhoPairing}) we known that $\mathscr{L}^{B}(\lambda)$ is a summand of $\mathsf{F}_{*}^{r}\Ox$ if and only if
\begin{equation}\label{OxPairing1}
0\leq \langle -\lambda,\alpha_{j}^{\vee}\rangle\leq\frac{2(p^{r}-1)}{p^{r}}\qquad(1\leq j\leq n).
\end{equation}
However, the paring $\langle -,-\rangle$ takes only integral values, and since $1\leq \frac{2(p^{r}-1)}{p^{r}}<2$, we may rewrite (\ref{OxPairing1}) as
\begin{equation}\label{OxPairing2}
\langle -\lambda,\alpha_{j}^{\vee}\rangle\in\{0,1\}\qquad(1\leq j\leq n).
\end{equation}
It is clear that $-\lambda$ satisfies (\ref{OxPairing2}) if and only if it is of form (\ref{SummandsOfOX}).
\end{ex}

\begin{rmk}
Let $X$ be a smooth projective variety. Recall that a vector bundle $\eE$ is a \textit{Frobenius summand} if it is a direct summand of $\mathsf{F}_{*}^{r}\Ox$ for some $r\geq0$ and that $X$ is of \textit{globally finite F-representation type} (for short: GFFRT) if the set of isomorphism classes of its Frobenius summands is finite. Among partial flag varieties, projective spaces, grassmannians $\textnormal{Gr}(2,n)$, and quadrics are known to be GFFRT (in the case of $
\PP^{n}$ this follows easily from the fact that $\mathsf{F}^{r}_{*}\mathscr{O}_{\PP^{n}}$ is a direct sum of line bundles. For the remaining two cases, see \cite{Readschelders-Spenko-VdBergh_2} and \cite{Langer}). It is an interesting problem to determine which partial flag varieties are GFFRT. Example \ref{Ox-example_1} shows that at least the number of line bundles that are Frobenius summands of $G/B$ is finite. An easy modification of this example shows that the same is true for all partial flag varieties.
\end{rmk}

\section{Frobenius kernels}

From now till the end of this paper, we work towards the proof of Theorem
\ref{MultiplicityofOxThm}. We keep the notation and the assumptions from the previous sections. In particular, $G$ is a semi-simple, simply connected algebraic group, and $X=G/P$ is a partial flag variety. It is also convenient to denote, for a positive integer $r$,
\begin{equation}\label{Xr}
X_{r}(T)\stackrel{\textnormal{def.}}{=}\left\{\mu\in X(T):0\leq\langle \mu,\alpha^{\vee}\rangle\leq p^{r}-1\textnormal{ for all }\alpha\in S\right\}.
\end{equation}
With this notation, the goal of this section is to prove the following lemma. In the next section, we use it to prove Theorem \ref{MultiplicityofOxThm}.

\begin{lem}\label{Injectivity_Lemma}
The evaluation map
$
\textnormal{ev}:H^{0}(\mu)\otimes\mathscr{O}_{G/P}\to\mathsf{F}^{r}_{*}\mathscr{L}^{P}(\mu)
$
is injective for every $\mu\in X_{r}(T)\cap X(P)$.
\end{lem}
\begin{rmk}\label{Frobenius_Affinity_Remark}
In the above, we make an identification $H^{0}(\mu)=H^{0}(X,\mathsf{F}^{r}_{*}\mathscr{L}^{P}(\mu))$ which follows from the affinity of $\mathsf{F}$. Although this identification is merely semi-linear ($\mathsf{F}$ is not a morphism of $K$-varieties), we have $\dim_{K}H^{0}(\mu)=\dim_{K}H^{0}(X,\mathsf{F}^{r}_{*}\mathscr{L}^{P}(\mu))$, because $K$ is algebraically closed (hence, perfect).
\end{rmk}

To prove Lemma \ref{Injectivity_Lemma} we use the $K$-linear Frobenius morphism. We have, for any $K$-variety $Y$, the pullback diagram
\[
\begin{tikzcd}
Y^{(r)}\arrow{r}{\theta_{r}}\arrow{d}&Y\arrow{d}\\
\textnormal{Spec }K\arrow{r}&\textnormal{Spec }K,
\end{tikzcd}
\]
where the bottom arrow is induced by the map $a\mapsto a^{p^{r}}$ on $K$. The projection $\theta_{r}$ is an isomorphism of schemes (but not $K$-schemes). In particular,
\begin{equation}\label{q-pushforward}
\theta_{r*}\mathscr{O}_{Y^{(r)}}=\mathscr{O}_{Y}.
\end{equation}
The absolute Frobenius $\mathsf{F}^{r}:Y\to Y$ induces, via the universal property of the fibered product, a map $\mathsf{F}^{(r)}:Y\to Y^{(r)}$, the $K$-\textit{linear Frobenius morphism}. From the definition,
\begin{equation}\label{q-F-composition}
q^{r}\circ \mathsf{F}^{(r)}=\mathsf{F}^{r}.
\end{equation}
Moreover, the association $Y\mapsto Y^{(r)}$ is functorial and commutes with products. Using these elementary properties it is easy to verify that $\mathsf{F}^{(r)}:G\to G^{(r)}$ is a morphism of group schemes over $K$. Moreover, if $X$ is a $G$-variety via $\mu:G\times X\to X$, then $X^{(r)}$ is a $G^{(r)}$-variety via $\mu^{(r)}:G^{(r)}\times X^{(r)}\to X^{(r)}$, and therefore also a $G$-variety via $\mu^{(r)}\circ(F^{(r)}\times \textnormal{id}_{X})$. Furthermore, $\mathsf{F}^{(r)}:X\to X^{(r)}$ is $G$-equivariant with respect to these actions. In particular, if $\eE$ is a $G$-equivariant vector bundle over $X$ then $\mathsf{F}^{(r)}_{*}\eE$ is $G$-equivariant in a natural way and the equality 
\[
H^{0}(X,\eE)=H^{0}(X^{(r)},\mathsf{F}^{(r)}_{*}\eE)
\]
holds in the category of $G$-modules. The group scheme
\[
G_{r}\stackrel{\textnormal{def.}}{=}\ker
\left(
\mathsf{F}^{(r)}:G\to G^{(r)}
\right)
\]
is called the $r$-th \textit{Frobenius kernel}. We refer the reader to \cite[I, Chapter 9 and II, Chapter 3]{Jantzen_book} for basic facts about $G_{r}$-modules. For $\mu\in X(T)$ we denote by $L(\mu)$ the unique irreducible $G$-module of the highest weight $\mu$. By \cite[II, Section 3.10, Proposition]{Jantzen_book} the irreducible $G_{r}$-modules are parametrized by $X_{r}(T)$ and given $\mu\in X_{r}(T)$ we write $L_{r}(\mu)$ for the irreducible $G_{r}$-module. Furthermore, by \cite[II, Section 3.15, Proposition]{Jantzen_book} we have an equality of $G_{r}$-modules
\begin{equation}\label{Simple_Gr-modules}
L(\mu)=L_{r}(\mu)\qquad(\mu\in X_{r}(T)).
\end{equation}
Given a $G$-module $M$ we can restrict it to $G_{r}$ and consider its socle $\textnormal{soc}_{G_{r}}M$. This socle is a $G$-module in a natural way \cite[II, Section 3.16]{Jantzen_book}. Moreover, by \cite[Formula 4.2]{Andersen_G_r}
\begin{equation}\label{Gr-socle}
\textnormal{soc}_{G_{r}}H^{0}(\mu)=L_{r}(\mu)\qquad (\mu\in X_{r}(T)).
\end{equation}
In particular, this socle is simple as a $G_{r}$-module by (\ref{Simple_Gr-modules}).
\medskip

Lemma \ref{Injectivity_Lemma} follows from the lemma below, which is essentially a generalization of the argument given by W.\ J.\ Haboush in his simple proof of Kempf's vanishing theorem \cite{Haboush_Kempf}.
\begin{lem}\label{Injectivity_lemma_2}
Let $\eE$ be a $G$-equivariant vector bundle over $X^{(r)}$ and let $V\subset H^{0}(X^{(r)},\eE)$ be a $G$-module. If $\textnormal{soc}_{G_{r}}V$ is a simple $G_{r}$-module then the evaluation map 
    $
    \textnormal{ev}_{V}:V\otimes  \mathscr{O}_{X^{(r)}}\to \eE
    $
is injective.
\end{lem}
\begin{proof} 
Let $W=\textnormal{soc}_{G_{r}}V$. By assumption, this is a simple $G_{r}$-module, which is also a $G$-module in a natural way. We have a canonical $G$-equivariant injection 
\[
\epsilon:W\otimes\mathscr{O}_{X^{(r)}}\to V\otimes\mathscr{O}_{X^{(r)}}.
\]
From the definition of the $G$-action on $X^{(r)}$ it follows that $G_{r}$ acts trivially on $X^{(r)}$, so the action of $G_{r}$ on any $G$-equivariant vector bundle preserves its fibers. In particular, the action of $G_{r}$ preserves the fibers of $V\otimes  \mathscr{O}_{X}$ and $\eE$, and the restriction of $\textnormal{ev}_{V}$ to every fiber is a morphism of $G_{r}$-modules. If $\textnormal{ev}_{V}$ is not injective then its restriction to some fiber is not injective on the $G_{r}$-socle of $V$ (hence it is zero on this socle by the simplicity). It follows that the composition
\[
\textnormal{ev}_{W}=\textnormal{ev}_{V}\circ\epsilon:W\otimes\mathscr{O}_{X^{(r)}}\to\eE
\]
is zero on some fiber. However, since $G$ acts transitively on $X$ and since $\mathsf F^{(r)}:G\to G^{(r)}$ is a surjection, we see that $G$ acts transitively on $X^{(r)}$. Since $\textnormal{ev}_{W}$ is $G$-equivariant, it follows from the above discussion that it is zero on every fiber. On the other hand, if $W\neq 0$ then $\textnormal {ev}_{W}$ is injective on the global sections, and hence is not zero. A contradiction.
\end{proof}

\begin{proof}[Proof of Lemma \ref{Injectivity_Lemma}]
First, because of (\ref{Gr-socle}), we may apply Lemma \ref{Injectivity_lemma_2} to $\eE=\mathsf{F}_{*}^{^{(r)}}\mathscr{L}(\mu)$ and $V=H^{0}(X^{(r)},\mathsf{F}_{*}^{^{(r)}}\mathscr{L}(\mu))=H^{0}(\mu)$ to obtain injection $H^{0}(\mu)\otimes\mathscr{O}_{X^{(r)}}\to \mathsf{F}_{*}^{^{(r)}}\mathscr{L}(\mu)$. The Lemma follows if we further pushforward this injection by $\theta_{r*}$ and use (\ref{q-pushforward}) and (\ref{q-F-composition}).
\end{proof}

\section{The multiplicity of $\Ox$}\label{Section: Multiplicity of Ox}

In this section, we prove Theorem \ref{MultiplicityofOxThm}.
We follow the notation and the assumptions from the previous sections. In particular, $G$ is a semi-simple, simply connected, algebraic group over $K$, $B\subset G$ is a fixed Borel subgroup, $B\subset P$ is a parabolic subgroup, and $X=G/P$. In the notation of (\ref{Xr}), the assumption of Theorem \ref{MultiplicityofOxThm} may be written more compactly as $\mu\in X_{+}(T)\cap X(P)$.

\begin{lem}\label{Injectivity_Lemma_3}
Let $\mu\in X_{r}(T)\cap X(P)$. Then there exists a positive integer $m$ and an injective map of $\Ox$-modules $\iota:\mathsf F^{r}_{*}\mathscr{L}(\mu)\hookrightarrow\Ox^{\oplus m}$.
\end{lem}

\begin{proof}
We use several times the elementary fact that the pushforward $f_{*}$ is left-exact for any morphism $f$. Recall the isomomorphism of Andersen--Haboush \cite{Andersen_Frobenius}, \cite{Haboush_Kempf}
\begin{equation}\label{A-H_isomorphism}
\mathsf{F}^{r}_{*}\mathscr{L}^{B}((p^{r}-1)\rho)\simeq\mathscr{O}_{G/B}^{\oplus m}\qquad (m=p^{r \dim G/B}).
\end{equation}
If $\mu\in X_{r}(T)$, then $(p^{r}-1)\rho-\mu$ is a dominant weight. Hence,
\[
\textnormal{Hom}_{\mathscr{O}_{G/B}}(\mathscr{L}^{B}(\mu),\mathscr{L}^{B}((p^{r}-1)\rho))=H^{0}((p^{r}-1)\rho-\mu)\neq0
\]
by (\ref{Non-zeroH0}). On $G/B$, every non-zero homomorphism from a line bundle is injective, so there exists an injective morphism
\[
\iota_{1}:\mathscr{L}^{B}(\mu)\hookrightarrow \mathscr{L}^{B}((p^{r}-1)\rho),
\]
and it follows from (\ref{A-H_isomorphism}) that we have an injective morphism
\[
\iota_{2}=\mathsf{F}_{*}^{r}\iota_{1}:\mathsf{F}^{r}_{*}\mathscr{L}^{B}(\mu)\hookrightarrow\mathscr{O}_{G/B}^{\oplus m}.
\]
Finally, we have a projection $\pi:G/B\to X$ such that $\pi_{*}\mathscr{L}^{B}(\mu)=\mathscr{L}^{P}(\mu)$ for all $\mu\in X(P)$. Since $\mathsf{F}^{r}\circ\pi=\pi\circ\mathsf{F}^{r}$ we obtain the desired injection
\[
\iota=\pi_{*}\iota_{2}:
\mathsf{F}^{r}_{*}\mathscr{L}^{P}(\mu)=
\pi_{*}\mathsf{F}^{r}_{*}\mathscr{L}^{B}(\mu)
\hookrightarrow
\pi_{*}\mathscr{O}_{G/B}^{\oplus m}=\Ox^{\oplus m}. \qedhere
\]
\end{proof}

We are now ready to give a proof of Theorem \ref{MultiplicityofOxThm}.

\begin{proof}[Proof of Theorem \ref{MultiplicityofOxThm}]
It follows from Lemma \ref{Injectivity_Lemma} that we have an injection
\[
\textnormal{ev}:H^{0}(\mu)\otimes\Ox\hookrightarrow\mathsf{F}^{r}_{*}\mathscr{L}(\mu),
\]
and from Lemma \ref{Injectivity_Lemma_3} we have an injection
\[
\iota:\mathsf F^{r}_{*}\mathscr{L}(\mu)\hookrightarrow\Ox^{\oplus m}.
\]
The composition $\iota\circ\textnormal{ev}$ is an injective morphism of globally free $\Ox$-modules. On a projective variety, every such morphism splits. Since $\iota\circ\textnormal{ev}$ splits, so does $\textnormal{ev}$.
\end{proof}

\begin{ex}
Any $\mu\in X_{+}(T)$ can be written uniquely as $\mu=\mu_{0}+p^{r}\mu_{1}$, with $\mu_{0}\in X_{r}(T)$ and $\mu_{1}\in X_{+}(T)$, and if, furthermore, $\mu\in X(P)$ then we also have $\mu_{0},\mu_{1}\in X(P)$. It follows from Theorem \ref{MultiplicityofOxThm} and the projection formula applied to $\mathsf{F}^{r}$ that $H^{0}(\mu_{0})\otimes\mathscr{L}(\mu_{1})$ is a direct summand of $\mathsf{F}^{r}_{*}\mathscr{L}(\mu)=\mathscr{L}(\mu_{1})\otimes\mathsf{F}_{*}^{r}\mathscr{L}(\mu_{0})$. If $\mu\in X_{r}(T)$, then we can determine another direct summand of $\mathsf{F}^{r}_{*}\mathscr{L}(\mu)$ by applying the above observation to
$\left(\mathsf{F}^{r}_{*}\mathscr{L}(\mu)\right)^{\vee}=\mathsf{F}^{r}_{*}\mathscr{L}(2(p^{r}-1)\rho-\mu)$
and dualizing.
\end{ex}

\begin{ex}\label{Ox-example_2}
Here is a special case of the previous example. Let $X=G/B$ be a full flag variety. It follows from Example \ref{Ox-example_1} that $\mathscr{L}^{B}(-\rho)$ is a direct summand of $\mathsf{F}^{r}_{*}\Ox$. We now compute its multiplicity. By the projection formula, this is the same as the multiplicity of $\Ox$ as as summand of
$
(\mathsf{F}_{*}^{r}\Ox)^{\vee}\otimes\mathscr{L}^{B}(-\rho)=\mathsf{F}_{*}^{r}\mathscr{L}^{B}((p^{r}-2)\rho).
$
By Theorem \ref{MultiplicityofOxThm} this multiplicity is $\dim_{K}H^{0}((p^{r}-2)\rho)$. In fact, the discussion from this and preceding sections shows that the evaluation map $H^{0}((p^{r}-2)\rho)\otimes\Ox\to\mathsf{F}^{r}_{*}\mathscr{L}^{B}((p^{r}-2)\rho)$ induces, after dualizing and twisting by $\mathscr{L}^{B}(-\rho)$, a natural surjection
\begin{equation}\label{surjection_1}
\mathsf{F}^{r}_{*}\Ox\to H^{0}((p^{r}-2)\rho)^{\vee}\otimes\mathscr{L}^{B}(-\rho)
\end{equation}
that is split in the category of $\Ox$-modules. For $r=1$ we have $L((p-2)\rho)^{\vee}=L((p-2)\rho)$ by \cite[II, Section 2.5, Proposition]{Jantzen_book}, so composing the natural surjection $H^{0}((p-2)\rho)^{\vee}\otimes\cO_{X}\to L((p-2)\rho)\otimes\cO_{X}$ with (\ref{surjection_1}) gives a natural surjection
\begin{equation}\label{surjection_2}
\mathsf{F}^{r}_{*}\Ox\to L((p-2)\rho)\otimes\mathscr{L}^{B}(-\rho).
\end{equation}
At the end of \cite{Gros-Kaneda}, Gros--Kaneda ask if (\ref{surjection_2}). This is indeed the case because already (\ref{surjection_1}) is split.
\end{ex}

\newpage

\bibliographystyle{plain}
\bibliography{Bibliography}

\end{document}